\RequirePackage{tikz}
\documentclass{arxiv}

\usepackage{thm-restate}

\usepackage{subcaption}
\usepackage{pgfplots}
\pgfplotsset{compat=1.18}
\usepackage{pgfplotstable}
\usepackage{colortbl}

\usepackage{lmodern}
\usepackage{anyfontsize}

\renewcommand{\paragraph}[1]{ {\bf #1.} }

\jyear{2023}%

\theoremstyle{thmstyleone}%
\newtheorem{theorem}{Theorem}
\newtheorem{proposition}[theorem]{Proposition}
\newtheorem{lemma}[theorem]{Lemma}%

\theoremstyle{thmstyletwo}%

\theoremstyle{thmstylethree}%

\begin{document}

\title[\footnotesize{An Efficient Framework for Global Non-Convex Polynomial Optimization with Nonlinear Polynomial Constraints}]{An Efficient Framework for Global Non-Convex Polynomial Optimization with Algebraic Constraints}


\author[1,2]{\fnm{Mitchell Tong} \sur{Harris}}\email{mitchh@mit.edu}
\author*[1]{\fnm{Pierre-David} \sur{Letourneau}}\email{pletourn@qti.qualcomm.com}
\author[1]{\fnm{Dalton} \sur{Jones}}\email{daltjone@qti.qualcomm.com}
\author[1]{\fnm{M. Harper} \sur{Langston}}\email{hlangsto@qti.qualcomm.com}

\affil[1]{\orgdiv{Qualcomm AI Research}\footnote{Qualcomm AI Research is an initiative of Qualcomm Technologies, Inc.}, \orgaddress{\street{5775 Morehouse Dr}, \city{San Diego}, \postcode{92121}, \state{CA}, \country{USA}}}
\affil[2]{Department of Mathematics, Massachusetts Institute of Technology, Cambridge, MA 02139}
\abstract{We present an efficient framework for solving algebraically-constrained global non-convex polynomial optimization problems over subsets of the hypercube. We prove the existence of an equivalent nonlinear reformulation of such problems that possesses essentially no spurious local minima.  Through numerical experiments on previously intractable global constrained polynomial optimization problems in high dimension, we show that polynomial scaling in dimension and degree is achievable when computing the optimal value and location.}
\blfootnote{The first author acknowledges support by the NSF Graduate Fellowship
under Grant No. 2141064.}



\maketitle
\section{Introduction}
\label{sec:introduction}
Due to the number of applications for which it can be leveraged, polynomial optimization is a longstanding and important problem in applied mathematics (e.g. \cite[Chapter 4]{Zhening2012Approx}). Works by Shor \cite{shor1987class}, Nesterov \cite{nesterov2000squared}, Parrilo \cite{parrilo2000structured} and Lasserre \cite{lasserre2001global} have provided tools for polynomial optimization that exploit polynomial structures and provide methods, which are more specialized than general nonlinear programming. In this work, we focus on making these tools for polynomial optimization more practical.

We concentrate on polynomial optimization over the hypercube and let $p, \{g^{(j)} \}_{j=1}^J$ be polynomials in $D$ variables and of degree at most $d$. We specifically consider the constrained non-convex polynomial optimization problem 
\begin{equation}\begin{aligned}
    \label{eq:problemcontinuous}
    \min_{x \in [-1,1]^D} & \;\; p(x) \\
    \mathrm{subject\, to} & \;\; g^{(j)}(x) = 0,\, j=1,...,J
\end{aligned}\end{equation}

A framework for efficiently solving \emph{unconstrained} polynomial optimization over the hypercube was proposed in \cite{Letourneau2023unconstrained}. Using inspiration from the ideas of Lasserre~\cite{lasserre2001global}, the authors demonstrated how such problems may be efficiently reformulated as problems involving a nonlinear objective over a convex cone of semi-definite matrices in $O(D\cdot d^2)$ dimensions, such that the reformulation possesses no spurious local minima (i.e., every local minimum is essentially a global minimum). 

This paper generalizes the reformulation presented in \cite{Letourneau2023unconstrained} to the case of constrained polynomial optimization problems over algebraic sets. Our generalized reformulation~\eqref{eq:reformulationcontinuous} is similar to that of~\cite{Letourneau2023unconstrained} with the addition of $J$ scalar constraints. This reformulation has similar guarantees. We have implemented an algorithm for solving the generalized reformulation in the \texttt{Julia} programming language and have produced positive numerical results.  Our contributions can be summarized as follows.
\begin{itemize}
    \item {\bf Efficient reformulations of algebraically-constrained polynomial optimization} We introduce a reformulation of~\eqref{eq:problemcontinuous}, which has a nonlinear objective with a feasible region, consisting of the intersection of the cone of semidefinite matrices, whose region is determined through scalar nonlinear constraints. This   reformulation is equivalent to the original problem in the sense that they share the same optimal value (Theorem \ref{thm:equivalence}) and optimal locations up to a canonical transformation (Theorem \ref{thm:globalmin}). The reformulation can be solved to optimality using descent techniques (due to the absence of local minima), an approach not generally possible with the original formulation.
    \item {\bf Efficient characterization of product measures over algebraic subsets of $[-1,1]^D$} We provide a complete and efficient characterization of the set of product measures supported over the set,
    $$ \left \{ x \in [-1,1]^D  \, : \, g^{(j)}(x) = 0, \, j=1,...,J \right \},$$ 
    via the moments of the measures using semidefinite and scalar nonlinear inequality constraints (Proposition \ref{prop:measure_extension_generalized}).
\end{itemize}

The remainder of this paper is structured as follows: Previous work is discussed in Section~\ref{subsec:previous_work}.  Section~\ref{sec:theory} introduces notation, theory, and our novel generalized reformulation of~\eqref{eq:problemcontinuous}. Section~\ref{sec:implementation} discusses algorithmic and technical aspects of our solver implementation in the \texttt{Julia} programming language. Section~\ref{sec:numerical} presents numerical results demonstrating correctness and performance of our framework. Section~\ref{sec:conclusion} presents conclusions and a discussion of future work. Proofs can be found in Appendix~\ref{sec:proofs}.

\subsection{Previous Work}
\label{subsec:previous_work}
While there are various approaches to polynomial optimization, this work relies heavily on the moment formulation for polynomial optimization introduced by Lasserre~\cite{lasserre2001global}. The moment formulation and its dual were connected to semidefinite programming (SDP) by \cite{lasserre2001global} and \cite{parrilo2000structured}, opening the door for practical computations. 

The idea of the moment formulation is as follows: rather than evaluating polynomials directly, the optimization is performed over \emph{measures}. The values of the objective and constraints are the integrals over the multidimensional measures, the moments of which are the integrals of the monomials in the original polynomial objective and constraints. Since using a finite-size semidefinite constraint to characterize a functional is necessary but not sufficient to ensure the moments actually come from a true measure, the SDP provides a relaxation to the original problem. 
The primary advantage of this approach is that the resulting problem is convex and can be solved relatively efficiently. However, rather than provide an exact reformulation,~\cite{lasserre2001global} results in a convex relaxation, where the solution provides a {\em lower bound} on the optimal minimum.  The quality of this bound depends on the degree of the Sums-Of-Squares (SOS)~\cite{doi:10.1137/040614141,prajna2002introducing} polynomials used to generate the relaxation. Increasing the size of the SDP provides better relaxations and bounds, but adversely affects scaling. In~\cite{lasserre2001global,Parrilo2003}, it was demonstrated that any sequence of problems of increasing size (``SOS degree'') will eventually converge to an exact solution in finite time (the sequence of problems is known as the ``Lasserre Hierarchy''). Although the hierarchy has been found to converge relatively rapidly on some specific problems (see, e.g., \cite{de2022convergence}), theoretical estimates indicate that super-exponential sizes are required in the worst case~\cite{Klerk2011}.



Degree, and ultimately scaling, estimates for the approach of~\cite{lasserre2001global}  rely on results by Putinar~\cite{putinar1993positive} and Nie and Schweighofer~\cite{nie2007complexity}, which express positive polynomials over compact semi-algebraic subsets of Euclidean spaces through SOS-like representations of finite degree, demonstrating that this can be done in any dimension.  However, the estimates demand a super-exponential SOS degree of the original problem in both dimension and degree, resulting in estimates suggesting intractable cost for solving polynomial optimization problems using~\cite{lasserre2001global}'s framework in the general case. 

The work of~\cite{Letourneau2023unconstrained} proposed a modification of~\cite{lasserre2001global} to improve scaling for the case of polynomial optimization over the hypercube. Instead of optimizing over general measures, the authors restrict to a sub-class of measures (the convex combination of product measures). The advantage is that characterizing such measures requires significantly smaller semidefinite constraints. The disadvantage is that the problem is no longer convex; nonetheless, there is always a guaranteed feasible descent direction starting from any non-optimal feasible point.

Our current work consists in a non-trivial generalization of these ideas in~\cite{Letourneau2023unconstrained}, extended to the case of polynomial optimization over the intersection of a general algebraic set and the hypercube. Our framework and results overcome the drawbacks of~\cite{lasserre2001global} in the context of constrained polynomial optimization by once again restricting the space of feasible 
measures. We show this restricted space can be represented efficiently in Proposition~\ref{prop:measure_extension_generalized}. In particular, we provide an explicit and efficient characterization of the set of product measures supported over, 
$$ \mathcal{S} :=  \left \{ x \in [-1,1]^D  \, : \, g^{(j)}(x) = 0, \, j=1,...,J \right \}.$$
In this generalized reformulation, we preserve the guarantees provided by~\cite{Letourneau2023unconstrained}. Further, our conclusions about efficiency are much stronger than the general case considered in~\cite{lasserre2001global}, highlighting the practicality of our generalized reformulation.

\section{Theory}
\label{sec:theory}
Let $p, \{g^{(j)} \}_{j=1}^J$ be polynomials in $D$ variables and of degree at most $d$. We consider the constrained non-convex polynomial optimization problem, 
\begin{equation*}\begin{aligned}
    \min_{x \in [-1,1]^D} & \;\; p(x) \\
    \mathrm{subject\, to} & \;\; g^{(j)}(x) = 0,\, j=1,...,J.
\end{aligned}\end{equation*}

\subsection{Preliminaries}
\label{sec:notation}
\paragraph{Polynomials} Let the \emph{dimension} of~\eqref{eq:problemcontinuous} be $D$, and let $d$ be the \emph{problem degree} (i.e., the maximum degree of the objective and constraint polynomials). A $D$-dimensional \emph{multi-index} $n=(n_1, n_2, ..., n_D)$ is a $D$-tuple of nonnegative integers. Let $\lvert n \rvert = \sum_{i=1}^D n_i$ be the \emph{multi-index degree}. A \emph{monomial} is written as,
\begin{equation}
    x^n = \prod_{i=1}^D x_i^{n_i}.
\end{equation}
Unless otherwise stated, all polynomials are written in the monomial basis.
The vector space of polynomials of degree $d$ in dimension $D$ is denoted $\mathbb{P}_{D,d}$; i.e., 
\begin{equation}
\label{eq:P_Dd}
    \mathbb{P}_{D,d} := \left \{ \sum_{\substack{\lvert n \rvert \leq d\\ \,n \in \mathbb{N}^D}} \, p_n \, x^n : \, p_n \in \mathbb{R} \right \},
\end{equation}
where the summation is over all multi-indices of degree less than or  equal to $d$. In particular, note that $\mathbb{P}_{D,d}$ has dimension
\begin{equation}
    \# \{ n \in \mathbb{N}^D \, : \, \lvert n \rvert \leq d \} = \binom{D+d}{d} \leq \min \left ( D^d, d^D \right ).
\end{equation}
Hence, for a fixed degree $d$, the dimension of the vector space of polynomials $\mathbb{P}_{D,d}$ grows \emph{polynomially} in the problem dimension $D$. The notation $\mathbb{P}_{D,\infty}$ is used to represent polynomials of any finite degree. The \emph{support of a polynomial} $p \in \mathbb{P}_{D,\infty}$ corresponds to the locations of its nonzero coefficients; i.e.,
\begin{equation}
    \mathrm{supp}(p) := \{ n \in \mathbb{N}^D \, : \, p_n \not = 0 \}, 
\end{equation}
and the cardinality (number of elements) of the support is denoted by
\begin{equation}
    N (p) := \# \, \mathrm{supp}(p).
\end{equation}
\paragraph{Moments of measures}
The set of regular Borel measures over $\mathbb{R}^D$ (\cite{cohn2013measure}) is denoted as $\mathbb{B}_{D}$, and
the \emph{support of a measure} $\mu(\cdot) \in \mathbb{B}_D$ is designated by $\mathrm{supp}(\mu(\cdot))$. The support represents the smallest set such that $\mu \left ( \mathbb{R}^D \backslash \mathrm{supp}(\mu(\cdot)) \right )  = 0$. 
For a fixed multi-index $n$, a $D$-dimensional Borel measure in $\mathbb{B}_{D}$ has an $n^{th}$ \emph{moment} defined to be the multi-dimensional Lebesgue integral,
\begin{equation}
    \mu_n := \int_{\mathbb{R}^D} x^n \, \mathrm{d} \mu(x).
\end{equation}
When referring to a measure, parentheses are employed (e.g., $\mu(\cdot)$). If the parentheses are absent (e.g., $\mu$), we are referring to a vector of moments associated with the measure. Let $X_i$ be measureable subsets of $\mathbb{R}$. A \emph{regular product measure} is a regular Borel measure of the form,
\begin{equation}
    \mu(X_1 \times \cdots \times X_D) = \prod_{i=1}^D  \mu_i(X_i) ,
\end{equation}
where each factor $\mu_i(\cdot)$ belongs to $\mathbb{B}_1$. In this context, the \emph{moments of a product measure} up to degree $d$ correspond to a $D$-tuple of real vectors $( \mu_1, \mu_2, ..., \mu_D ) \in \mathbb{R}^{(2d+1) \times D}$,\footnote{We consider the first $2d$, rather than $d$, moments for reason that will become apparent in Proposition \ref{prop:measure_extension_generalized} below.} where each element is defined as,
\begin{equation}
    \mu_{i,n_i} := \int_{\mathbb{R} } x_i^{n_i} \, \mathrm{d} \mu_i (x_i).
\end{equation}
The vectors $\mu_i$ refer to the vector of \emph{moments of the 1D measure} $\mu_i(\cdot)$.  With a single sub-index, $\mu_i \in \mathbb{R}^{2d+1}$ represents the vector of moments, and with two sub-indices, $\mu_{i,n_i} \in \mathbb{R}$  represents scalar moments (vector elements) of a 1D measure.  In this work, we are interested in collections of $L$ measures that are each products of 1D measures. Those moments are all represented by the $(2d+1) \times D \times L$ real numbers $\mu = \{\mu_i^{(l)}\}_{l,i}$. Finally, we introduce the following notation to represent the \emph{moments of sums of product measures}:
\begin{equation}
    \label{eq:sumprodmeasure}
    \phi_n (\mu)  = \sum_{l=1}^L  \prod_{i=1}^D \mu^{(l)}_{i,n_i}.
\end{equation}
\paragraph{Moment matrices}
Let $\mu_i$ be a sequence of moments of a 1D measure. The \emph{degree-$d$ 1D moment matrix} (or \emph{moment matrix}) associated with a polynomial $h \in \mathbb{P}_{1,d^\prime}$ is defined to be the $(d +1) \times (d  +1) $ matrix $\mathcal{M}_d(\mu_i, h)$ with entries,
 \begin{equation}
     [ \mathcal{M}_d(\mu_i, h) ]_{m,n} = \int h(x_i) \, x_i^{m + n} \, \mathrm{d} \mu_i(x_i) = \sum_{k_i=0}^{d^\prime} h_{k_i} \, \mu_{k_i+m+n},
 \end{equation}
for $0 \leq m, n \leq d $. When $g(\cdot ) \equiv 1$, we use the short-hand notation,
\begin{equation}
    \mathcal{M}_d(\mu_i; 1) = \mathcal{M}_d(\mu_i).
\end{equation}
The construction of such matrices requires the knowledge of only the first $2 d+ d^\prime + 1$ moments of $\mu_i (\cdot) $; i.e., $\{ \mu_{i,n}\}_{n=0}^{2d + d^\prime} $. 

\subsection{Reformulation on the hypercube}
\label{sec:reformulation_original}

This section summarizes results presented in \cite{Letourneau2023unconstrained} that underly this work. In \cite{Letourneau2023unconstrained}, problems of the form,
\begin{equation}
\label{eq:problemcontinuous_orig}
    \min_{x \in [-1,1]^D } \, p(x),
\end{equation}
were treated. That work proposed the following reformulation of Problem~\eqref{eq:problemcontinuous_orig}:
\begin{equation}
\begin{aligned}
\label{eq:reformulationcontinuous}
    \min_{  \mu \in \mathbb{R}^{(2d+1)\times D \times L} }&\;\; \sum_{n \in \mathrm{supp}(p)} \, p_n \, \phi_n (\mu)  \\
    \mathrm{subject\, to} &  \;\; \mathcal{M}_d (\mu^{(l)}_i) \succeq 0,  \\
    &\;\; \mathcal{M}_{d-1} (\mu^{(l)}_i; \, 1-x_i^2) \succeq 0,  \;   \\
    &\;\; \phi_{(0,...,0)}  (\mu)  = 1,  \\
     &\phi_n (\mu)  = \sum_{l=1}^L  \prod_{i=1}^D \mu^{(l)}_{i,n_i},
\end{aligned} 
\end{equation}
where the moments $\phi_n (\mu) $ are defined in Equation~\eqref{eq:sumprodmeasure}, and the constraints are over all $i=1,\dots, D$ and $l=1,\dots, L$. It was shown in~\cite{Letourneau2023unconstrained} that the reformulated problem is equivalent to the original in that they share the same global optimal value and locations (under a canonical correspondence). This paper presents a generalization of this reformulation, where polynomial constraints are added to Problem~\ref{eq:problemcontinuous_orig} (Section \ref{sec:reformulation_general}). 

The proof of the aforementioned equivalence relies heavily on \cite[Proposition 1]{Letourneau2023unconstrained}, which states that to every element in the feasible set of Problem \ref{eq:reformulationcontinuous}, there exists a corresponding Borel measure with moments $\{ \phi_n (\mu) \}$ for all $0 \leq n_i \leq d$; i.e., the multi-dimensional moment problem can be solved efficiently over the hypercube when sums of product measures are concerned.

One contribution of this paper is a generalization of Proposition 1 from \cite{Letourneau2023unconstrained}, expanding the result to the moment problem over algebraic subsets of $[-1,1]^D$. This result is stated in Proposition \ref{prop:measure_extension_generalized} below.

\subsection{Generalized reformulation for polynomial constraints}
\label{sec:reformulation_general}
Under our formulation, algebraic constraints on the support of finite measures are enforced by requiring that for all $j$,
\begin{equation}
    \int   \left ( g^{(j)}(x) \right )^2 \mathrm{d}\mu(x) = 0.
\end{equation}
When satisfied, this constraint guarantees that the measure $\mu(\cdot)$ is supported over the set,
\begin{equation}
    \mathcal{S} :=  \left \{ x \in [-1,1]^D  \, : \, g^{(j)}(x) = 0, \, j=1,...,J \right \}.
\end{equation}
The formal statement and proof of this claim are presented in the Proposition~\ref{prop:positivepolysupport} (Appendix \ref{sec:proofs}). This also forms the basis for Proposition~\ref{prop:measure_extension_generalized} where we characterize product measures with such a support from their moments.

Having intuitively established the rationale behind the characterization of product measures over elgebraic sets, we may now introduce our generalized formulation: let $\gamma^{(j)}$ be the vector of coefficients of the polynomial $ \left ( g^{(j)}(\cdot) \right )^2$ in a monomial basis. The generalized reformulated problem can be stated as
\begin{align}
\begin{split}
\label{eq:reformulationcontinuous_generalized}
    \min_{  \mu \in \mathbb{R}^{(d+1)\times D \times L} }&\;\; \sum_{n \in \mathrm{supp}(p)} \, p_n \, \phi_n (\mu)  \\
    \mathrm{subject\, to} &  \;\; \mathcal{M}_d (\mu^{(l)}_i) \succeq 0,  \\
    &\;\; \mathcal{M}_{d-1} (\mu^{(l)}_i; \, 1-x_i^2) \succeq 0,  \;\;\;  \\
    &\;\; \gamma^{(j)} \cdot \phi(\mu^{(l)}) = 0\\
    &\;\; \phi_{(0,...,0)}  (\mu)  = 1 \\
   & \phi_n (\mu)  = \sum_{l=1}^L  \prod_{i=1}^D \mu^{(l)}_{i,n_i},
\end{split} 
\end{align}
where the constraints are over all $i=1,\dots, D$,  $l=1,\dots, L$, and  $j=1,\dots,J$, $\cdot$ indicates inner product, and $\phi(\mu^{(l)})$ is a vector of moments of the $l^{\mathrm{th}}$ product measure; i.e., $\phi_n(\mu^{(l)}) = \prod_{i=1}^D \mu^{(l)}_{i,n_i} $ .

This reformulated Problem~\eqref{eq:reformulationcontinuous_generalized} is equivalent to Problem~\ref{eq:problemcontinuous} in that it shares the same optimal value and optimal locations under a canonical map (see Theorems~\ref{thm:equivalence} and~\ref{thm:globalmin} below). These results are a consequence of the following Proposition~\ref{prop:measure_extension_generalized}, which states that for any moment vector satisfying the constraints of the reformulation, there is a measure whose support is contained in the original feasible region.

\begin{restatable}{proposition}{measureextensiongeneralized}
\label{prop:measure_extension_generalized}
    Let $D,d \in \mathbb{N}$ and $\mu := ( \mu_1, \mu_2, ..., \mu_D ) \in \mathbb{R}^{(2d+1) \times D}$ be such that for each $i = 1, ..., D$, $\mu_{i,0} = 1$, and
    \begin{align*}
        \mathcal{M}_d(\mu_i) &\succeq 0,\\
        \mathcal{M}_{d-1}(\mu_i;  1-x_i^2) &\succeq 0 .
    \end{align*}
    and in addition for every $j=1, ..., J$,
    \begin{equation}
        \label{eq:newconstr}
        \gamma^{(j)} \cdot \phi(\mu)  = 0,
    \end{equation}
    where $ \gamma^{(j)}$ is the vector of polynomial coefficients or $\left ( g^{(j)}(x) \right )^2$. Then there exists a regular Borel product measure,
    \begin{equation*}
        \mu( \cdot ) := \prod_{i=1}^D  \mu_i(\cdot), 
    \end{equation*}
    supported over the algebraic set
    \begin{equation}
        \label{eq:def_semialgebraic}
        \mathcal{S}:= \cap_{j=1}^J \left \{ x \in [-1,1]^D \, : \, g^{(j)}(x) = 0 \right \} 
    \end{equation}
    such that $\mu \left ( \mathcal{S} \right )  = 1$. Further, for all multi-indices $n \in \mathbb{N}^D$ such that for every $i =1,\dots, D$, $0 \leq n_i  \leq d$, we have
    \begin{equation*}
         \int_{[-1,1]^D} \, x^n \, \left (  \prod_{i=1}^D \mu_i(x_i) \right ) \, \mathrm{d} x =  \prod_{i=1}^D \mu_{i,n_i}.
    \end{equation*}
\end{restatable}

The equivalence of the generalized formulation with Problem \ref{eq:problemcontinuous} carries verbatim from \cite{Letourneau2023unconstrained} as stated in the following theorems. A discussion of the proofs may be found in Appendix \ref{sec:proofs_analoguous}. The first theorem states that the global minima is unchanged, and the second theorem provides a characterization of global optimality. 
\begin{restatable}{theorem}{equivalence}
    \label{thm:equivalence}
    Problem \ref{eq:problemcontinuous} and Problem \ref{eq:reformulationcontinuous_generalized} share the same global minimum value.
\end{restatable}
\begin{restatable}{theorem}{globalmin}
    \label{thm:globalmin}
    Let $L=2$, and let
    $$ \mu := \left \{  (\mu_1^{(l)}, ..., \mu_D^{(l)})  \right \}_{l=1}^{2} \in \mathbb{R}^{(2d+1) \times D \times 2 } $$ 
    be a feasible point of Problem~\ref{eq:reformulationcontinuous_generalized}. This point corresponds to a global minimum of Problem \ref{eq:problemcontinuous} if and only if for any $l \in \{0, 1\}$ such that $\alpha^{(l)} := \prod_{i=1}^D \mu_{i,0}^{(l)} \not = 0$,
    the point,
    \begin{eqnarray}
        \left \{  \left ( \frac{\mu_1^{(l)}}{\alpha^{(l)}}, \mu_2^{(l)}, ..., \mu_D^{(l)} \right ) ,  (0,...,0)  \right \}
    \end{eqnarray}
    is a local minimum. The corresponding global minimum of Problem \ref{eq:problemcontinuous} corresponds to the sum of product measures with such moments following Proposition \ref{prop:measure_extension_generalized}.
\end{restatable}

The proof of these theorems is entirely analogous to those found \cite{Letourneau2023unconstrained}. A discussion on how to (trivially) adapt the latter to the case at hand can be found in Appendix \ref{sec:proofs_analoguous}.

\subsection{Semi-Algebraic Constraints}
Though we do not explicitly deal with semi-algebraic constraints in this paper, this section briefly discusses the treatment of such constraints within our framework. Semi-algebraic constraints are those of the form,
\begin{eqnarray}
    h^{(k)} (x) \geq 0,
\end{eqnarray}
where $h^{(k)} \in \mathbb{P}_{d,D}$ are some polynomials for $k=1,...,K < \infty$. Polynomial optimization problems involving such constraints take the form,
\begin{equation}\begin{aligned}
    \label{eq:problemcontinuoussemi}
    \min_{x \in [-1,1]^D} & \;\; p(x) \\
    \mathrm{subject\, to} & \;\; g^{(j)}(x) = 0,\, j=1,...,J \\
    & \;\; h^{(k)}(x) \geq 0,\, k=1,...,K .
\end{aligned}\end{equation}
One simple and efficient way to address this problem is to convert it to a problem of the form \ref{eq:problemcontinuous} through the introduction of $K$ slack variables $\{ y_k \}_{k=1}^K \in \mathbb{R} $. In this case, one may rewrite Problem \ref{eq:problemcontinuoussemi} as,
\begin{equation}
\begin{aligned}
\label{eq:problemcontinuoussemiequiv}
    \min_{x \in [-1,1]^D, y\in \mathbb{R}^K} & \;\; p(x) \\
    \mathrm{subject\, to} & \;\; g^{(j)}(x) = 0,\, &j=1,...,J \\
    & \;\; h^{(k)}(x) - y_k^2 = 0,\, &k=1,...,K, 
\end{aligned}
\end{equation}
transforming it into the prescribed form discussed in this paper.

\section{Algorithm and Implementation}
\label{sec:implementation}
In this section, we briefly describe how we numerically solve the generalized reformulated Problem~\ref{eq:reformulationcontinuous_generalized}, presented in Section~\ref{sec:reformulation_general}. The reformulation consists of a nonlinear semidefinite program; i.e., an optimization problem with a nonlinear objective, semi-definite constraints and additional linear and nonlinear scalar constraints. We now describe how we handle the semidefinite constraints, followed by details on our numerical implementation.

\subsection{Handling Semidefinite Constraints}
Semi-definite constraints are denoted by: $\mathcal{M} \succeq 0$. Within our implementation, we handle semi-definite constraints  using the Burer-Monteiro (BM) method \cite{burer2003nonlinear}, enforcing them through equality constraints of the form: $\mathcal{M}_d = XX^T$ for some $X \in \mathbb{R}^{(d+1) \times r}$. In our case, the rank $r$ was generally chosen to be $r=d+1$ (full rank). The use of the BM method comes at the cost of introducing (second-order) nonlinearities, but such a reformulation does not change the feasible set when the rank is chosen as such; indeed, Boumal et al.~\cite{boumal2016non} have shown that when the rank of a solution is sufficiently small, every second-order stationary point of a SDP treated through the BM framework corresponds to a global minimum.

With this choice, the reformulation becomes
\begin{align}
\begin{split}
\label{eq:reformulationcontinuous_generalized_BM}
    \min_{  \mu \in \mathbb{R}^{(d+1)\times D \times L} }&\;\; \sum_{n \in \mathrm{supp}(p)} \, p_n \, \phi_n (\mu)  \\
    \mathrm{subject\, to} &  \;\; \mathcal{M}_d (\mu^{(l)}_i) = X^{(l)}_i {X_i^{(l)}}^T  \\
    &\;\; \mathcal{M}_{d-1} (\mu^{(l)}_i; \, 1-x_i^2) = Y^{(l)}_i {Y_i^{(l)}}^T\\
    &\;\;  X^{(l)}_i \in \mathbb{R}^{ (d+1) \times (d+1) }, \; Y^{(l)}_i \in \mathbb{R}^{ d \times d } \\
    &\;\; \gamma^{(j)} \cdot \phi(\mu) = 0   \;\;\;\;\;  \;\;\;\;\;\;\;\;\;\;\;\;\;\; i=1,..., D, l=1,..., L , j=1,...,J \\
    &\;\; \phi_{(0,...,0)}  (\mu)  = 1,  \\
    & \;\;  \phi_n (\mu)  = \sum_{l=1}^L  \prod_{i=1}^D \mu^{(l)}_{i,n_i}. 
\end{split} 
\end{align}
 
This problem now only contains scalar equality constraints, making it more amenable to existing numerical solvers. Table \ref{tab:cost} describes the computational cost associated with the evaluation of the objectives and constraints. Note that this cost is \emph{at most polynomial in dimension and degree}.

\begin{table}[htbp!]
    \centering
    \begin{tabular}{|c|c|c|c|} \hline
         Expression  &   Size & Eval. Cost \\ \hline
         $\sum_{n \in \mathrm{supp}(p)} \, p_n \, \phi_n (\mu) $  & $1$ & $ O\left ( N(p) \, L \, D \right )$  \\ \hline
         $ \mathcal{M}_d \left ( \mu^{(l)}_i \right ) $ &  $(d+1) \times (d+1)$ & $  O\left ( L \, D \, d^2 \right )$  \\ \hline
         $ \mathcal{M}_{d-1} \left (\mu^{(l)}_i; \, 1-x_i^2 \right )    $  & $ d \times d$ & $  O\left ( L \, D \, d^2 \right )$  \\ \hline
         $ X_i^{(l)} {X_i^{(l)}}^T $, $ Y_i^{(l)} {Y_i^{(l)}}^T $  & $ O (d) \times O(d)$ & $  O\left ( D\, d^3 \right )$  \\ \hline
         $  \sum_{l=1}^L \prod_{i=1}^D \mu^{(l)}_{i,0} - 1 $ &  $1$ & $  O\left ( L \, D \right )$  \\ \hline
          $ \gamma^{(j)} \cdot \phi(\mu)$ & $J$ & $  O\left (  \max_j N^2 ( g^{(j) } ) \, J \, D \right )$   \\ \hline
    \end{tabular}
    \caption{Size and evaluation costs of all parts of each component of the proposed generalized reformulation.}
    \label{tab:cost}
\end{table}

\subsection{Numerical Implementation Details}
To solve this problem numerically, we have implemented a solver in the \texttt{Julia} programming language~\cite{bezanson2012julia}. For optimization purposes, we used the \texttt{Nonconvex.jl} package~\cite{Tarek2023nonconvex}, which provides a convenient interface for connecting to many backend solvers and is well-suited for the treatment of nonlinear equalities and inequalities.

For the backend solver, we chose to use the \texttt{ipopt} algorithm implementation~\cite{wachter2006implementation} with first-order approximation. We used default parameter values, with the exception of the overall convergence tolerance, which corresponds to the maximum scaled violation under the KKT formulation (i.e., Eq.(6),~\cite{wachter2006implementation}), that we set between $\epsilon = 10^{-2}$ and $\epsilon = 10^{-1}$. 

Polynomials found in the objectives and the constraints were represented using \texttt{Julia}'s \texttt{DynamicPolynomials.jl}  package~\cite{legat2021polynomials}, which offers an efficient symbolic way of expressing such functions. Finally, to compute the gradient information required by the \texttt{ipopt} algorithm, we \texttt{Julia}'s \texttt{Forwarddiff.jl} package~\cite{RevelsLubinPapamarkou2016}, which provides an efficient implementation of automatic differentiation. 

\section{Numerical experiments}
\label{sec:numerical}
In this section, we present numerical results demonstrating the performance of our proposed algorithm and software. To display the versatility and power of our technique, we consider two families of constrained polynomial optimization problems that highlight the hurdles often encountered using traditional descent methods, specifically (1) local minima, (2) multi-modality, as well as (3) disconnected and discrete feasible regions.  \emph{Our reformulation overcomes these challenges}.

The first family of problems is treated in Section~\ref{sec:numerical:annulus} and consists of a concave quadratic objective over a connected non-convex feasible (algebraic) set, resulting in four (4) local minima within the feasible region regardless of dimension (see Figure~\ref{fig:ellipticalannulus2d}). The second family of problems is discussed in Section~\ref{sec:numerical:disconnected} and involves a concave (non-convex) objective over a discrete feasible region consisting of exponentially many points (i.e., an irregular lattice, see Figure~\ref{fig:disjoint_patches2D}).

To demonstrate the difference between the traditional and reformulated approaches, we utilize a single nonlinear interior point solver (\texttt{ipopt}) to solve the problem both in its original form (Equation~\eqref{eq:problemcontinuous}) as well as its reformulated form (Equation~\eqref{eq:reformulationcontinuous_generalized_BM}). We demonstrate that, while the original formulation suffers from non-convexity and generally fails to find a global optimum, the reformulated approach consistently succeeds (Tables~\ref{table:ellipticalannulussolved} and \ref{table:patchessolved}).

\subsection{Elliptical Annulus}
\label{sec:numerical:annulus}
\paragraph{Problem setup} For our first numerical experiments, we consider the problem of minimizing a concave quadratic objective over a feasible region consisting of two concentric spheres as shown in Figure~\ref{fig:ellipticalannulus2d}. This problem can be expressed using algebraic constraints as
\begin{align}
    \label{eqn:ellipticalannulus}
    \min \;\; &p(x) := -\left ( x_1 - 0.1 \right )^2 \\
    \mathrm{s.\,t.} \;\; &   g(x) := \left ( \sum_{i=1}^D x_i^2 - 1 \right ) \, \left (   \sum_{i=1}^D x_i^2 - 0.5^2 \right) = 0.
\end{align}
In particular, note that the objective consists of a shifted concave quadratic function varying in the first dimension only. An illustration of the problem when $D = 2$ is presented in Figure~\ref{fig:ellipticalannulus2d}.
\begin{figure}[htbp!]
\centering
\includegraphics[width=.8\textwidth]{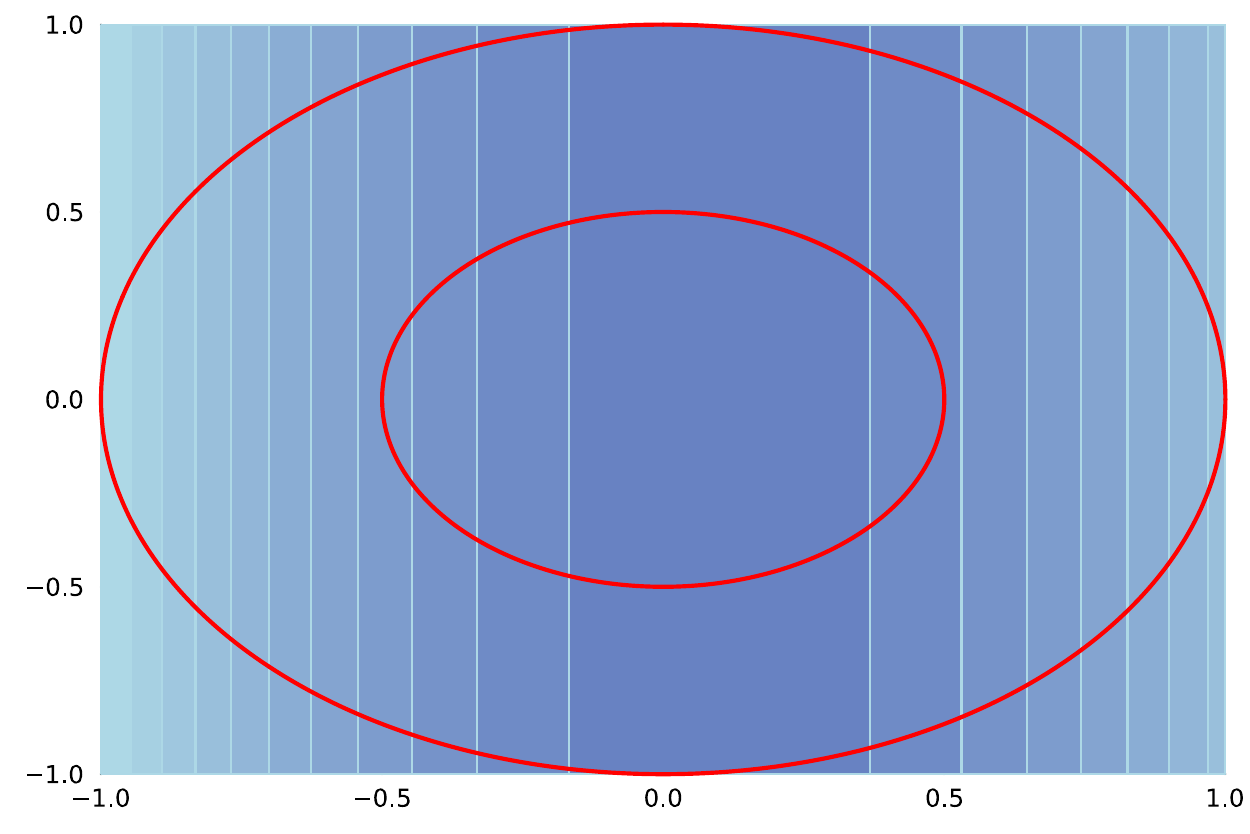}
    \caption{Representation of Problem \ref{eqn:ellipticalannulus} in two dimensions (2D). The feasible region consists of  non-convex spherical shells. The objective is concave in the first dimension, constant in remaining dimensions, and possesses four (4) local minima over the feasible set.}
    \label{fig:ellipticalannulus2d}
\end{figure}

In any given dimension $D$, the problem contains exactly four (4) local minima, only one of which is a global minimum located at $x^* = (-1, 0, \dots, 0)$ with value $p(x^*) = -1.21$.  Within the non-convex feasible region, there are four (4) basins of attractions (corresponding to each local minima) of roughly equal volume.  Given a starting point uniformly chosen within the feasible region, we thus expect traditional descent methods to succeed at most $25\%$ of the time.

\paragraph{Numerical results}
For our numerical experiments, we solved four (4) instances of Problem~\eqref{eqn:ellipticalannulus} with different random starting points for each dimension $D \in \{ 2, 3 \dots, 22 \}$. The results are reported below.

Figure~\ref{fig:elliptical-annulus-objerror} shows the average optimal objective value error (over $4$ instances) for each dimension under consideration. It is observed that the reformulation succeeds in finding the global minimum value in every instance; specifically, the relative error is on average under $10^{-4}$, well below the expected bound (using an overall convergence tolerance of $\epsilon = 10^{-2}$).
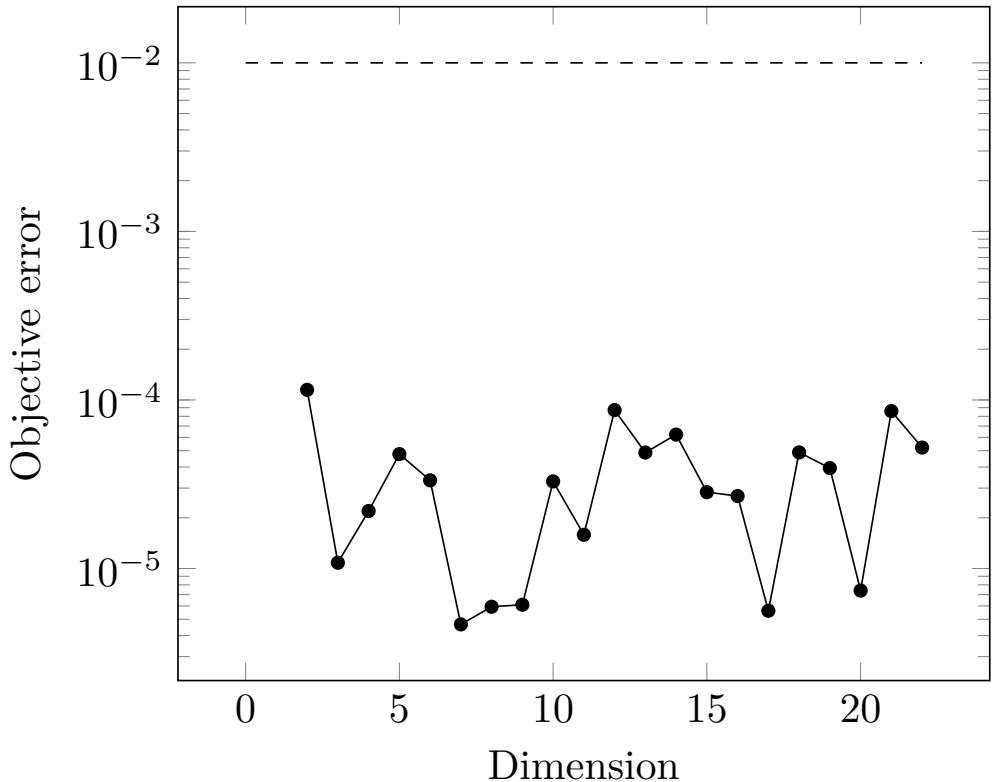
\begin{figure}[!htbp]
\begin{subfigure}{.8\textwidth}
    \centering 
\resizebox{\textwidth}{!}{\begin{tikzpicture}
\begin{axis}[
  xlabel=Dimension,
  ylabel=Objective error,
  ymode = log,
  legend = northwest,]

\addplot+[mark options={black, scale=0.75}, black] table[x=dimension,y=objectivefmgpsrelerrmean, col sep = comma] {data/paper-demo-ellipse-algebraic-summary-0-processed.csv};

 \addplot[mark=none, black, dashed, samples=2, domain = 0:22] {0.01};
\end{axis}
\end{tikzpicture}}
\end{subfigure}
\caption{Average relative errors in spherical shells problem. The error is on average approximately $10^{-4}$ and lies below the solver threshold ($10^{-2}$) regardless of the dimension $D$.}\label{fig:elliptical-annulus-objerror}
\end{figure}

The average wall time taken by our numerical solver (Section~\ref{sec:implementation}) for computing the solution of the aforementioned problems is shown in Figure~\ref{fig:elliptical-annulus-timings}. We observe polynomial scaling with dimension, which is in line with expectations (see Table~\ref{tab:cost}). Indeed, as discussed in Section~\ref{sec:reformulation_general}, the efficiency of our scheme is such that the reformulated problem lies in a space for which dimension grows slowly (linearly) with the dimension of the underlying problem. In this sense, the most computationally expensive operations involve the evaluation of the objective, the gradients, and the constraints. In this context, the cost scales as $O(D^5)$, which is nearly the observed scaling in Figure~\ref{fig:elliptical-annulus-timings}.
\begin{figure}[!htbp]
    \centering 
\begin{tikzpicture}
\begin{axis}[
    xtick={
        2,4,8,16
    },
          xticklabel={
        \pgfkeys{/pgf/fpu=true}
        \pgfmathparse{exp(\tick)}%
        \pgfmathprintnumber[fixed relative, precision=3]{\pgfmathresult}
        \pgfkeys{/pgf/fpu=false}
      },
  xlabel=Dimension,
  ylabel=Wall time (s),
  ymode = log,
  xmode = log,
  legend pos = north west,
  xmin = 2,
  xmax = 22
]

\addplot+[mark options={black, scale=0.75}, black] table[x=dimension,y=timesfmgpsmean, col sep = comma] {data/paper-demo-ellipse-algebraic-summary-0-processed.csv};
\addlegendentry{Reformulation}

\addplot[mark=none, black, dashed, samples=2, domain = 2:32] {x^(5)};
 \addlegendentry{$\mathcal{O}(D^{5})$}

\end{axis}
\end{tikzpicture}
\caption{Average wall times of solving reformulated elliptical annulus problem using our novel formulation (solid line). The dotted line indicates a scaling of $O(D^{5})$.}
\label{fig:elliptical-annulus-timings}
\end{figure}
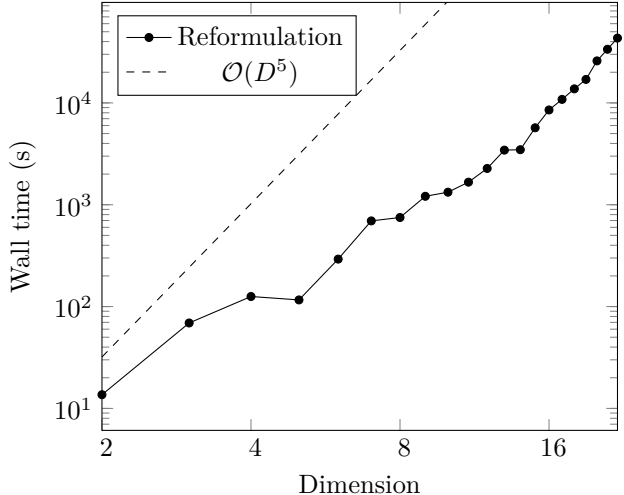

\paragraph{Comparison with original formulation}
\label{sec:comparison_formulations_ellipse}
To demonstrate the superior characteristics of our reformulation, we compare it to the results obtained using the original formulation. For this purpose, we leverage \texttt{ipopt}~\cite{wachter2006implementation}, an existing, robust solver commonly used for nonlinear optimization. This is specifically employed for solving the original formulation (Problem~\ref{eq:problemcontinuous}) directly within an overall convergence tolerance of $10^{-1}$, followed by utilizing the same solver for the reformulated problem (Problem~\ref{eq:reformulationcontinuous_generalized_BM}).
Table~\ref{table:ellipticalannulussolved} displays the percentage of problems for which the global optimum was found using our novel reformulation (center column) versus the original formulation (right column) as a function of dimension $D$. As shown, while the original formulation fails approximately $50\%$ of the time, our reformulation approach is consistently successful $100\%$ of the time in our experiments. As discussed, due to the problem's geometry possessing four basin of attractions around four local minima (only one of which is also global), this outcome is expected. In fact, this is a common problem that plagues all solver based on local descent when the problem is not convex or possesses more than one local minimum with different values. In this case, local descent solvers (including \texttt{ipopt}) may only find local minima which lie within the basin of attraction of the starting point. In our case, however, we leveraged a commonly-employed random initialization to derive different starting points for each of the $4$ runs of the same problem.  Moreover, while the existence of four basins of attractions of roughly equal sizes explains the low success rate of the original formulation, since the reformulated problem using our approach does not possess spurious local minima, we do not suffer from this common drawback.

\begin{table}[htbp!] \centering
\begin{filecontents*}{data/paper-demo-ellipse-algebraic-summary-0-processed.csv}
dimension,nrow,objectivefmgpsrelerrmean,objectiveipoptrelerrmean,minlocerrormean,timesfmgpsmean,ipoptsolvemean,fmgpssolvemean,objectivefmgpsrelerr_std,objectiveipoptrelerrstd,minlocerrorstd,timesfmgpsstd,ipoptsolvestd,fmgpssolvestd
2,4,0.00011484913952169308,0.08264633135035503,0.010527635624238564,13.63083228625,0.75,1.0,0.00020919271100318347,0.1652886640965993,0.007676818360653695,3.711215321650393,0.5,0.0
3,4,1.082633154498344e-5,0.1652938413249292,0.056069960922835865,69.01649064225,0.5,1.0,1.589134575790689e-5,0.19085487735525003,0.0354060432926325,23.939378261219044,0.5773502691896257,0.0
4,4,2.1948579768088736e-5,9.031715504499108e-7,0.0688232040288227,125.56921078900001,1.0,1.0,4.024855792799978e-5,7.253247736155821e-7,0.10106391636535154,61.3520098223219,0.0,0.0
5,4,4.773882869311791e-5,0.2479348718476136,0.08609929757247076,116.2413807375,0.25,1.0,9.477374429705946e-5,0.16528919134614636,0.04345659348946072,24.410311661278904,0.5,0.0
6,4,3.336149024771651e-5,0.2479347067613192,0.06251437858224126,292.56195340275,0.25,1.0,6.550430027991758e-5,0.16528901220045084,0.02476236131647766,59.12216203812805,0.5,0.0
7,4,4.66049273099076e-6,0.16528971130620526,0.038848615642528145,694.07646428975,0.5,1.0,5.274242846888677e-6,0.1908594484992696,0.030245440692303723,596.6429883153133,0.5773502691896257,0.0
8,4,5.933309143772313e-6,0.1652895648602361,0.07047773545100043,749.24431745025,0.5,1.0,7.369523756036676e-6,0.19085969737246317,0.04381523143937834,89.52516367973085,0.5773502691896257,0.0
9,4,6.097492971508558e-6,0.08264525404096242,0.07626627825823162,1210.6414627240001,0.75,1.0,1.040517952359972e-5,0.1652890189464907,0.04848765300652593,344.5371919125306,0.5,0.0
10,4,3.287106471552997e-5,0.24793435838109007,0.06492545460499652,1328.307210451,0.25,1.0,3.6602654148972e-5,0.16528886286128122,0.007974812013553514,600.3892164261274,0.5,0.0
11,4,1.5853138931259074e-5,0.2479339666878094,0.09525888913978728,1666.2755213194998,0.25,1.0,1.9083798180042792e-5,0.1652891662328256,0.0630613282054743,326.5411928852293,0.5,0.0
12,4,8.706615620181293e-5,0.16529043589147302,0.08331909159029775,2271.22119961125,0.5,1.0,6.338644685250398e-5,0.19086013171581673,0.02876518751088315,361.04457559316995,0.5773502691896257,0.0
13,4,4.875572361152649e-5,0.3305787282195523,0.0851111000284133,3437.48922573225,0.0,1.0,5.4553757048333464e-5,3.6051607316471064e-7,0.05357251334373246,1244.8347294933526,0.0,0.0
14,4,6.232586494974672e-5,0.33057900907767723,0.10609068954197522,3470.3413922495,0.0,1.0,9.844343203492294e-5,4.827623066153015e-7,0.030212422879570988,443.0858998857753,0.0,0.0
15,4,2.839020133390482e-5,0.08264464588833956,0.0810652377104299,5696.99187966925,0.75,1.0,2.046839241043819e-5,0.16528924953883845,0.03438467189246281,809.0405180590635,0.5,0.0
16,4,2.692592911858001e-5,0.24793392854677465,0.12845368219659636,8526.036046907,0.25,1.0,3.1135147625003075e-5,0.16528927952837735,0.07818135127322334,1624.268163077087,0.5,0.0
17,4,5.615088876050825e-6,0.0826446481929845,0.08699631661435966,10827.475288112251,0.75,1.0,8.144782123920677e-6,0.16528923741321375,0.04268457775664328,2270.176263891164,0.5,0.0
18,4,4.886285657454162e-5,2.020291732346459e-7,0.08028718321201761,13741.2153253095,1.0,1.0,4.5965022598660944e-5,3.4596927625872054e-7,0.03753162240268096,1051.6292558382017,0.0,0.0
19,4,3.943328340380879e-5,0.24793394196970703,0.1262132919390844,17005.62600468975,0.25,1.0,3.472590237324467e-5,0.1652892844592978,0.04584190248740467,3155.7560585280175,0.5,0.0
20,4,7.399651307809099e-6,0.08264484479184862,0.07126105303499856,25860.46632168875,0.75,1.0,6.973972524578254e-6,0.16528910333932018,0.04708190969279642,2978.0696896675317,0.5,0.0
21,4,8.589437317716093e-5,0.08264487936602198,0.07872803671315651,33636.51608138025,0.75,1.0,0.00015181847829493875,0.16528915225414015,0.07716529150685512,9283.172145121343,0.5,0.0
22,2,5.219805471429058e-5,1.3436309640659581e-8,0.12616951584474528,43297.0362307705,1.0,1.0,7.276944288127222e-5,6.115671936928344e-9,0.018719838135383242,5131.3386247843955,0.0,0.0
\end{filecontents*}
\pgfplotstabletypeset[every even row/.style={
before row={\rowcolor[gray]{0.9}}},
every head row/.style={
before row=\toprule,after row=\midrule},
every last row/.style={
after row=\bottomrule}, col sep = comma,
	columns={dimension, [index]7, [index]6},
columns/dimension/.style={column name=$D$},
columns/fmgpssolvemean/.style={column name=Reformulation, 
        preproc/expr={100*##1},
        postproc cell content/.append style={
            /pgfplots/table/@cell content/.add={}{\%},
        }
        },
columns/ipoptsolvemean/.style={column name=Original,
preproc/expr={100*##1},
        postproc cell content/.append style={
            /pgfplots/table/@cell content/.add={}{\%},
        }}]{data/paper-demo-ellipse-algebraic-summary-0-processed.csv}
\caption{Proportion of problems for which our Reformulation approach (center) and the Original formulation (right) find the global minimum of the elliptic annulus problem (Problem~\ref{eqn:ellipticalannulus}) within an error of $10^{-2}$. Every problem was solved $4$ times using \texttt{ipopt} as the backend.}
\label{table:ellipticalannulussolved}
\end{table}


\subsection{Discrete Feasible Region}
\label{sec:numerical:disconnected}
\paragraph{Problem setup}
For our second numerical experiment, we chose a more difficult family of problems consisting of a concave objective and discrete feasible set of the form,
\begin{equation}
\begin{aligned}
\label{eqn:disjointpatches}
    \min_{x \in \mathbb{R}^D} & \qquad  p(x) := - \sum_{i=1}^D (x_i + 0.1)^2,  \\
    \textrm{s.t.} & \qquad  g^{(i)}(x) = \left (x_i + \frac{1}{3} \right ) \cdot x_i \cdot \left (x_i - \frac{2}{3} \right ) = 0\; \;  i=1, ..., D .
\end{aligned}
\end{equation}
The objective and feasible region are depicted in two dimensions in Figure~\ref{fig:disjoint_patches2D}. 
\begin{figure}[!htbp]
 \centering   \includegraphics[width=.75\textwidth]{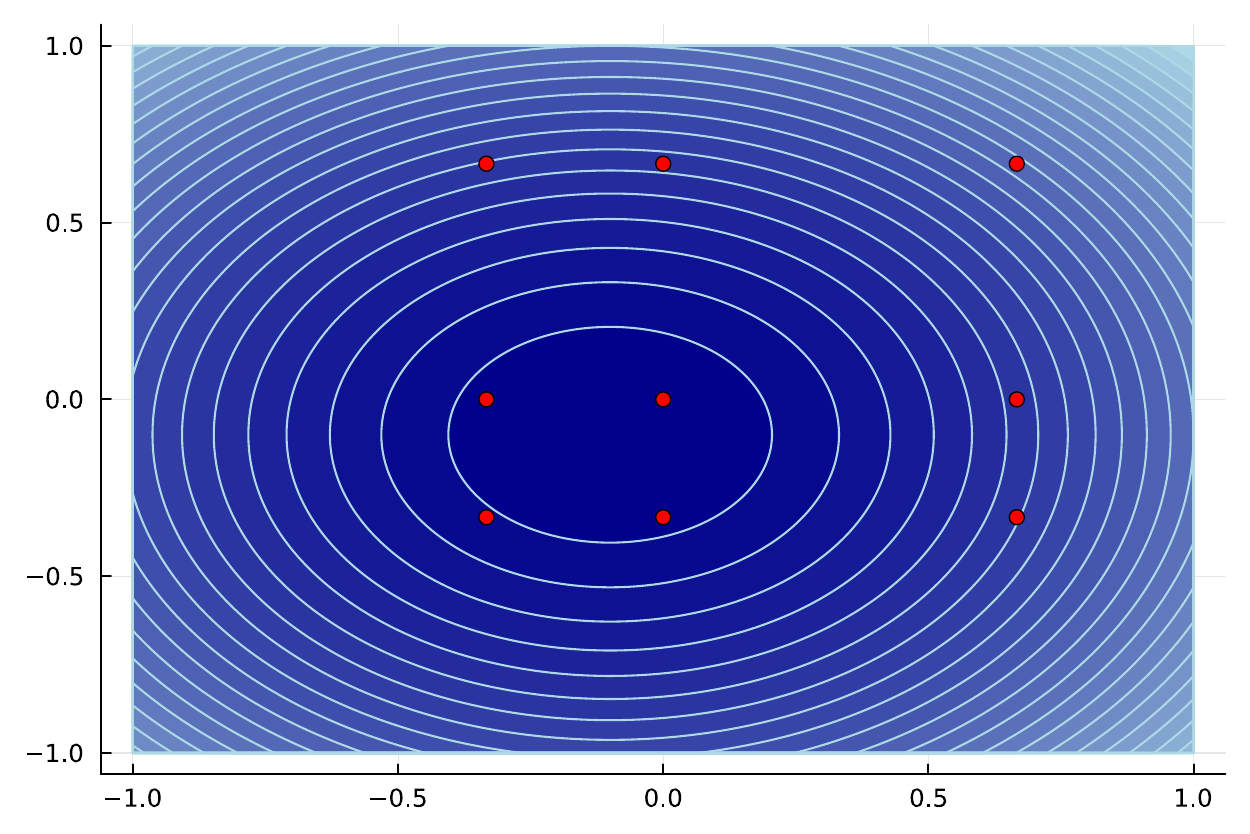}
    \caption{Two dimensional version of discrete non-convex optimization problem.}
    \label{fig:disjoint_patches2D}
\end{figure}
The objective is a concave quadratic function, and the feasible region is comprised of an exponential number of points (disconnected regions) with $3^D$ feasible points in dimension $D$. The true global minimum of Problem~\eqref{eqn:disjointpatches} is located at $x^* = \left( \frac{2}{3}, \frac{2}{3}, \dots, \frac{2}{3} \right )$ and has value: $p(x^*) = -\frac{529}{900} \cdot D$. 

\paragraph{Numerical results} We carried out numerical experiments where we solved four (4) instances of Problem~\eqref{eqn:disjointpatches} in both the reformulated (Problem~\ref{eq:reformulationcontinuous_generalized_BM}) and original (Problem~\ref{eq:problemcontinuous}) forms; different uniformly random points for dimensions $D \in \{ 2, \dots, 21 \}$ were generated as starting positions. We discuss results below.

Figure~\ref{fig:patches-objerror} shows the objective error found by the reformulated problem as a function of dimension. Note that the error is consistently below $10^{-1}$, within solver tolerance, demonstrating that our reformulation-based solver is indeed correct and solves Problem~\eqref{eqn:disjointpatches} in every instance regardless of dimension.

\begin{figure}[!htbp]
    \centering 
    \begin{subfigure}{.8\textwidth}
\resizebox{\textwidth}{!}{\begin{tikzpicture}
\begin{axis}[
  xlabel=Dimension,
  ylabel=Objective error,
  ymode = log,
  legend pos = south east,
  xmin = 1,
  xmax = 21]
\addplot+[mark options={black, scale=0.75}, black, error bars/.cd, 
        y dir=both, 
        y explicit] table[x=dimension,y=objectivefmgpsrelerrmean, 
        col sep = comma] {data/paper-demo-discrete-algebraic-summary-0-processed.csv};
 \addplot[mark=none, black, dashed, samples=2, domain = 0:21] {0.1};
\end{axis}
\end{tikzpicture}}
\end{subfigure}
\caption{Average relative objective errors in disjoint patches problem. The error is on average approximately $10^{-2}$ and lies below the solver threshold ($10^{-1}$) regardless of the dimension $D$.}
\label{fig:patches-objerror}
\end{figure}
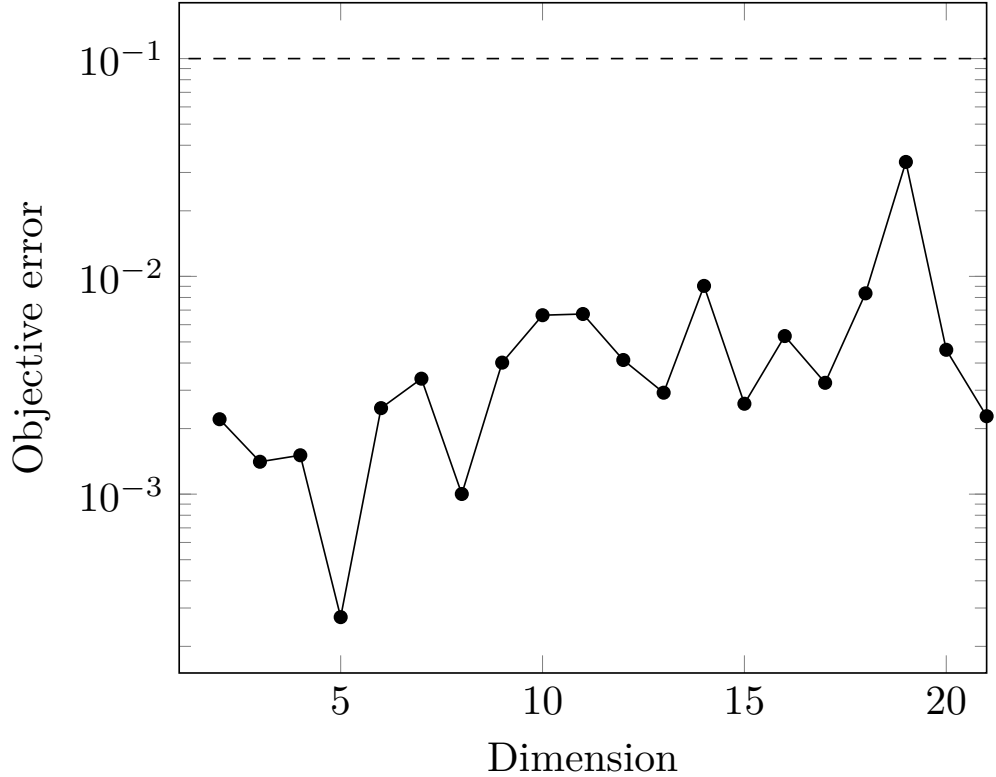

Next, Figure~\ref{fig:patches-timings} shows time complexity / scaling results. We observe polynomial scaling is achieved  as a function of dimension corresponding to $O(D^4)$. This is in line with the most computationally expensive operations in Table~\ref{tab:cost}; in this case, this corresponds to the evaluation of the objective, resulting in a total cost of $O(D^3)$ per iteration. 

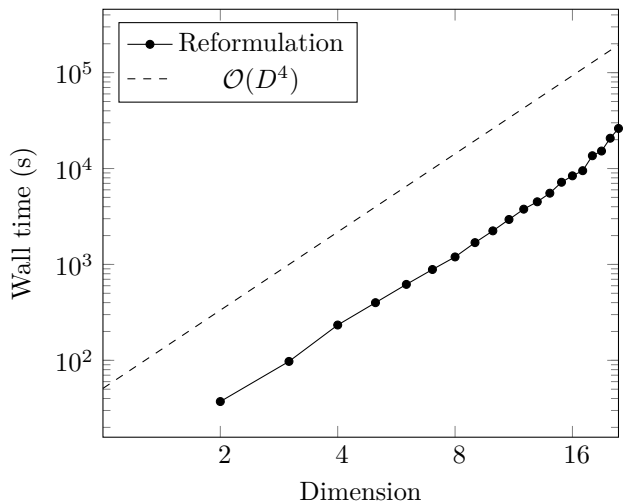
\begin{figure}[!htbp]
    \centering 
\begin{tikzpicture}
\begin{axis}[
    xtick={
        2,4,8,16
    },
      xticklabel={
        \pgfkeys{/pgf/fpu=true}
        \pgfmathparse{exp(\tick)}%
        \pgfmathprintnumber[fixed relative, precision=3]{\pgfmathresult}
        \pgfkeys{/pgf/fpu=false}
      },
  xlabel=Dimension,
  ylabel=Wall time (s),
  ymode = log,
  xmode = log,
  legend pos = north west,
  xmin = 1,
  xmax = 21
]

\addplot+[mark options={black, scale=0.75}, black] table[x=dimension,y=timesfmgpsmean, col sep = comma] {data/paper-demo-discrete-algebraic-summary-0-processed.csv};
 \addlegendentry{Reformulation}

\addplot[mark=none, black, dashed, samples=2, domain = 1:21] {50 + x^4};
 \addlegendentry{$\mathcal{O}(D^4)$}

\end{axis}
\end{tikzpicture}
\caption{Average wall times of solving reformulated disjoint patches problem using our novel formulation (solid line). The dotted line indicates a scaling of $O(D^4)$.}
\label{fig:patches-timings}
\end{figure}

\paragraph{Comparison of novel reformulation with original formulation}
In this section, we again compare the solutions to the reformulated Problem~\eqref{eq:reformulationcontinuous_generalized_BM} to the solution of the original Problem~\eqref{eq:problemcontinuous}), using the same process as described in Section~\ref{sec:comparison_formulations_ellipse} in the context of Problem~\ref{eqn:disjointpatches}. The results for this case are shown in Table~\ref{table:patchessolved}. 

We observe that, while the original formulation (right column) generally fails, our proposed reformulation (center column) succeeds in identifying the global minimum $100\%$ of the time.  We note that we observe some success for the original formulation in low dimension and consistent failure in higher dimension; this can be explained as follows: as previously discussed, the feasible region possesses exponentially (in dimension) many points, only one of which is a global minimum. To solve the problem in its original form, the interior point solver first computes a starting point in the original domain. With random initialization, we can expect the starting point to be more or less uniformly distributed over the discrete feasible set. Once such a point has been found, the interior point solver uses a local descent method to find a local minimum (and so terminates immediately). This implies that the global minimum will be found using the original formulation if and only if the starting point lies at the global minimum. However, based on the previous reasoning, this occurs with exponentially small probability in dimension. Thus, as the dimension grows, we expect the original formulation to fail consistently, while it has low, though non-trivial, probability of succeeding in lower dimensions (as observed).

These results empirically highlight the power of our reformulation in solving difficult non-convex problems for which current optimization techniques fail.

\begin{table}[htbp!]\centering
\begin{filecontents*}{data/paper-demo-discrete-algebraic-summary-0-processed.csv}
dimension,nrow,objectivefmgpsrelerrmean,objectiveipoptrelerrmean,minlocerrormean,timesfmgpsmean,fmgpssolvemean,ipoptsolvemean,objectivefmgpsrelerrstd,objectiveipoptrelerrstd,minlocerrorstd,timesfmgpsstd,fmgpssolvestd,ipoptsolvestd
2,4,0.0022082148973300497,0.5763718243153071,0.0013556319029120777,37.183256833499996,1.0,0.25,0.0010632637362639119,0.4334786048040574,0.000697815223908247,4.78605439295108,0.0,0.5
3,4,0.0014083030010327829,0.31578818976309747,0.0010833969399567964,97.445373491,1.0,0.5,0.0021928675722756815,0.3611049639024339,0.0012766688456149601,4.3134593910409755,0.0,0.5773502691896257
4,4,0.0015084908745634732,0.579948331046327,0.0011390294041265722,233.490972074,1.0,0.0,0.0008212212190411637,0.12671277097149306,0.0007903134653316473,1.1711674130894085,0.0,0.0
5,4,0.0002724282271037388,0.5992065624160503,0.0006118694350654976,399.34505389775,1.0,0.0,0.00021819125993989286,0.3254119574129691,0.0005195596504084211,0.3469873742101597,0.0,0.0
6,4,0.002483197948568957,0.7667004786703728,0.0021489928922452752,619.3232425715,1.0,0.0,0.0012385114524707033,0.11932558788328503,0.0007575020176666634,0.9400707582002595,0.0,0.0
7,4,0.0033898035347710857,0.5983500091902152,0.002554597774812953,885.46717022975,1.0,0.0,0.0014537020258587083,0.08130108963789771,0.0015214855099931514,9.886026664984604,0.0,0.0
8,4,0.0010016177952057738,0.7266978170386927,0.0013601073413379375,1197.27832599475,1.0,0.0,0.00029910108869785084,0.06308405605514551,0.0004520056240515326,81.705288097317,0.0,0.0
9,4,0.004016763594611724,0.6444137406644272,0.003208418210401181,1688.495114286,1.0,0.0,0.0037042134117433235,0.051498634549924115,0.0026359194615718797,2.574959501183677,0.0,0.0
10,4,0.006632625085921437,0.607622849026394,0.011640495098033898,2239.5820715207497,1.0,0.0,0.009890072282968907,0.122932978020144,0.020044674825350023,2.684314250365673,0.0,0.0
11,4,0.0067178298368824425,0.67195108094785,0.0045898096872585194,2942.8797655575,1.0,0.0,0.004436292819717961,0.18772040794714953,0.0028530193478157515,51.02650961584637,0.0,0.0
12,4,0.004130208538434908,0.5973745114592286,0.003297316280483339,3769.66826880975,1.0,0.0,0.0025077998262833722,0.1796988195560548,0.0019693739813281866,5.3358264799212245,0.0,0.0
13,4,0.0029227120996782405,0.6071690894947,0.002433831810901778,4500.3649838702495,1.0,0.0,0.0010265626182479118,0.169497459341753,0.0007213854761532497,255.60671817654594,0.0,0.0
14,4,0.009041014017570754,0.5973583440508429,0.006306319915981621,5532.631142190499,1.0,0.0,0.013244100875436244,0.09691683500913066,0.007587176659875241,570.4720994542673,0.0,0.0
15,4,0.0026007269708241644,0.5245510312614603,0.0027701158810581475,7193.503239927751,1.0,0.0,0.00257759296229223,0.14778305606182793,0.0017221640700028397,24.740961838265015,0.0,0.0
16,4,0.005320080250267558,0.5931839321971177,0.004327647848492194,8395.68861661325,1.0,0.0,0.0043829076663910605,0.0588914074971176,0.0031298006599175966,1256.0284814804345,0.0,0.0
17,4,0.0032451935671063154,0.5763995876429198,0.0035866121321556314,9497.19353128675,1.0,0.0,0.001673421417316419,0.0273228885705717,0.002672045580492013,2773.5366765806875,0.0,0.0
18,4,0.00835721178146049,0.5477002638844848,0.006260892757220053,13597.5795215085,1.0,0.0,0.004741125550077613,0.029082547250726464,0.0023577423788847993,255.1607477486261,0.0,0.0
19,4,0.03353097235820363,0.6022558738535926,0.02005310716238443,15218.313200995251,1.0,0.0,0.035821163014324836,0.0779710342009513,0.020587980599334837,2505.0959371636322,0.0,0.0
20,4,0.004600437997900811,0.5864451796832437,0.007303852487094249,20698.188861368748,1.0,0.0,0.0020460495945894804,0.08125538970232002,0.0062480085261819645,44.06719677490607,0.0,0.0
21,3,0.0022798481941327738,0.5021409818215364,0.002330556251424126,26178.201929919,1.0,0.0,0.001320409531247241,0.1460327910626748,0.001342748342623193,1436.6835215145059,0.0,0.0
\end{filecontents*}
\pgfplotstabletypeset[every even row/.style={
before row={\rowcolor[gray]{0.9}}},
every head row/.style={
before row=\toprule,after row=\midrule},
every last row/.style={
after row=\bottomrule}, col sep = comma,
	columns={[index]0, [index]6, [index]7},
columns/dimension/.style={column name=$D$},
columns/fmgpssolvemean/.style={column name=Reformulation, 
        preproc/expr={100*##1},
        postproc cell content/.append style={
            /pgfplots/table/@cell content/.add={}{\%},
        }
        },
columns/ipoptsolvemean/.style={column name=Original,
preproc/expr={100*##1},
        postproc cell content/.append style={
            /pgfplots/table/@cell content/.add={}{\%},
        }}]{data/paper-demo-discrete-algebraic-summary-0-processed.csv}
\caption{Percentage of problems for which our reformulation (center) and the original formulation (right) find the global minimum of the disjoint patches problem (Problem~\ref{eqn:disjointpatches}) within an error of $10^{-1}$. Each problem was solved $4$ times using \texttt{ipopt} as a  backend.}
\label{table:patchessolved}
\end{table} 

\section{Conclusion}
\label{sec:conclusion}
We have introduced a novel reformulation of general constrained polynomial optimization as nonlinear problems with essentially no spurious local minima. With an implementation of the reformulated problem in the \texttt{Julia} programming language, we have tested and observed the correctness of the approach. Furthermore, using difficult, previously-intractable constrained polynomial optimization problems, we have presented evidence of superior performance and practicality compared to existing techniques.

From a theoretical standpoint, future work will target the treatment of stationary points (i.e., suboptimal points) at which, from a numerical standpoint, the reformulated objective may find itself in a feasible region with a relatively flat descent landscape. We also intend to consider performance improvements, both algorithmic and from a software optimization perspective. This includes parallelism on shared memory machines and GPUs accelerations on heterogeneous architectures.

\backmatter

\newpage
\appendix
\begin{appendices}

\section{Proofs}
\label{sec:proofs}


The goal of this section is to prove Proposition~\ref{prop:measure_extension_generalized}, which requires a few intermediate results. To begin, let $D \in \mathbb{N}$, $x \in \mathbb{R}^D$ and $g(x)$ be some measurable function (in our case, a polynomial), $n \in \mathbb{N}$, and 
\begin{align}
    \mathcal{G} := \left \{ x \in \mathbb{R}^D \, : \, g(x) =0 \right \} \\
    A_n := \left \{ x \in \mathbb{R}^D \, : \, g^2(x) \geq \frac{1}{n} \right \}. 
    \label{eqn:An}
\end{align}
For all that follows, we assume that $\mathcal{G} \not = \emptyset$. 


\begin{lemma}
    \label{lemma:set_convergence}
    Consider the set $A_n$ defined in \eqref{eqn:An} and let $\mathcal{S} \subset \mathcal{G}^c$ be a measurable set. Then
    \begin{equation}
        A_n \cap \mathcal{S}  \uparrow_{n}^\infty \mathcal{S} 
    \end{equation}
    monotonically, and
    \begin{equation}
        \mu\left ( A_n \cap \mathcal{S} \right )  \rightarrow_{n}^\infty  \mu( \mathcal{S} ).
    \end{equation}
\end{lemma}
\begin{proof}
    First, we show that the set $A_n \cap \mathcal{S}$ converges to the set $\mathcal{S}$ by first considering
    \begin{align}
        \liminf_n A_n \cap \mathcal{S} &= \cup_{n \geq 1} \cap_{k \geq n} A_k \cap \mathcal{S} \\
        &= \cup_{n \geq 1} A_n \cap \mathcal{S} \\
        &= \mathcal{S}
    \end{align}
    by construction of the sets. Similarly,
    \begin{align}
        \limsup_n A_n \cap \mathcal{S} &= \cap_{n \geq 1} \cup_{k \geq n} A_k \cap \mathcal{S} \\
        &= \cap_{n \geq 1} \mathcal{S} \\
        &= \mathcal{S},
    \end{align}
    demonstrating convergence. The fact that the convergence is monotonic follows from $A_m \subset A_n$ for all $n \geq m$ by definition. In particular, this implies point-wise monotonic convergence of the following indicator functions,
    \begin{equation}
        \mathbb{I}_{A_n \cap S} (x) \uparrow_{n=1}^\infty \mathbb{I}_{ S} (x).
    \end{equation}
    Therefore, the monotone convergence theorem implies that
    \begin{equation}
        \mu\left ( A_n \cap S\right ) =  \int \mathbb{I}_{A_n \cap S} (x) \, \mathrm{d} \mu(x)   \rightarrow_{n}^\infty  \int \mathbb{I}_{S} (x) \, \mathrm{d} \mu(x)   =  \mu\left (  S\right ) ,
    \end{equation}
    and the result follows.
\end{proof}

The following proposition shows a sufficient condition for the support of the measure to be constrained to an algebraic variety. The rest of the reformulation is a way of encoding this sufficient constraint. 
\begin{proposition}
    \label{prop:positivepolysupport}
     Let $\mu(\cdot)$ be a finite measure supported on $\mathbb{R}^D$ and assume that,
     \begin{equation}
         \int g^2(x) \, \mathrm{d} \mu(x) = 0.
     \end{equation}
     Then,
     \begin{equation}
         \mathrm{supp}(\mu(\cdot)) \subset \left \{ x \in \mathbb{R}^D \, : \, g(x) =0 \right \} =: \mathcal{G}.
     \end{equation}
\end{proposition}
\begin{proof}
    Proceed by contradiction. Assuming the statement does not hold, there must exist a measurable set $\mathcal{S} \subset \mathcal{G}^c$ such that
    \begin{equation}
        \mu\left ( \mathcal{S} \right ) > 0.
    \end{equation}
    Consider the integral,
    \begin{align}
         \int_{\mathbb{R}^D} g^2(x) \, \mathrm{d} \mu(x) &=  \int_{\mathcal{S}^c} g^2(x) \, \mathrm{d} \mu(x) +  \int_{A_n \cap \mathcal{S}} g^2(x) \, \mathrm{d} \mu(x)  + \int_{A_n^c \cap \mathcal{S}} g^2(x) \, \mathrm{d} \mu(x) \\
         &\geq \int_{A_n \cap \mathcal{S}} g^2(x) \, \mathrm{d} \mu(x) ,
    \end{align}
    for each $n \in \mathbb{N}$, where the inequality follows from the non-negativity. of the integrand Now, by Lemma \ref{lemma:set_convergence}, there exists $N \in \mathbb{N}$ such that for every $n \geq N$
    \begin{equation}
        \mu\left ( A_n \cap \mathcal{S} \right ) > \frac{\mu( \mathcal{S} )}{2} > 0.
    \end{equation}
    Fix $n=N$ and consider
    \begin{align}
        \int_{\mathbb{R}^D} g^2(x) \, \mathrm{d} \mu(x) &\geq \int_{A_N \cap \mathcal{S}} g^2(x) \, \mathrm{d} \mu(x)  \\
        &\geq  \frac{1}{N} \, \mu \left ( A_N \cap \mathcal{S} \right )\\
        &\geq \frac{1}{N} \, \frac{\mu( \mathcal{S} )}{2}\\
        & > 0,
    \end{align}
    where we used the definition of $A_N$ in the second inequality. This is a contradiction and therefore $\mu(\mathcal{S}) = 0$ for all measurable subsets of $\mathcal{G}^c$, i.e., $\mu(\cdot)$ must be supported on $\mathcal{G}$. 
\end{proof}

Finally, we recall the following result from \cite{Letourneau2023unconstrained},
\begin{proposition}[Proposition 1, \cite{Letourneau2023unconstrained}]
\label{prop:measureextension}
    Let $D,d \in \mathbb{N}$ and $( \mu_1, \mu_2, ..., \mu_D ) \in \mathbb{R}^{(2d+1) \times D}$ be such that for each $i = 1, ..., D$, $\mu_{i,0} = 1$, and
    \begin{align*}
        \mathcal{M}_d(\mu_i) &\succeq 0,\\
        \mathcal{M}_{d-1}(\mu_i;  1-x_i^2) &\succeq 0 .
    \end{align*}
    Then, there exists a regular Borel product measure,
    \begin{equation*}
        \mu( \cdot ) := \prod_{i=1}^D  \mu_i(\cdot), 
    \end{equation*}
    supported over $[-1,1]^D$ such that $\mu \left ( [-1,1]^D \right )  = 1$, and
    \begin{equation*}
         \int_{[-1,1]^D} \, x^n \, \left (  \prod_{i=1}^D \mu_i(x_i) \right ) \, \mathrm{d} x =  \prod_{i=1}^D \mu_{i,n_i}
    \end{equation*}
    for all multi-index $n \in \mathbb{N}^D$ such that $0 \leq n_i  \leq d$ for all $i$.
\end{proposition}

We are now ready to prove the main technical proposition from Section~\ref{sec:reformulation_general}:

\measureextensiongeneralized*

\begin{proof}
    
To begin, following our hypotheses, Proposition~\ref{prop:measureextension} implies the existence of a product measure,
\begin{equation*}
    \mu( \cdot ) := \prod_{i=1}^D  \mu_i(\cdot), 
\end{equation*}
supported over $[-1,1]^D$ such that $\mu\left( [-1,1]^D \right ) =1$. Then, note that by construction, Equation~\eqref{eq:newconstr} is  equivalent to
\begin{eqnarray}
    \gamma^{(j)} \cdot \phi(\mu) = \int \left ( g^{(j)}(x) \right )^2 \, \mathrm{d} \mu(x) = 0.
\end{eqnarray}
Therefore, it follows from Proposition \ref{prop:positivepolysupport} that $\mu(\cdot)$ must be supported over
\begin{equation}
    \left \{ x \in [-1,1]^D \, : \, g^{(j)}(x) = 0 \right \}.
\end{equation}
However, since this is trus for every $j=1,...,J$, we conclude that $\mu(\cdot)$ must be supported on $\mathcal{S}$.

\end{proof}

\subsection{Proofs of Theorem \ref{thm:equivalence} and Theorem \ref{thm:globalmin} }
\label{sec:proofs_analoguous}

As mentioned in Section \ref{sec:reformulation_general}, the proof of Theorem \ref{thm:equivalence} and Theorem \ref{thm:globalmin} are entirely analogous to those of Theorem 2 and Theorem 3 from  \cite{Letourneau2023unconstrained}. Indeed, in the current context, it suffices to replace the sets $\mathcal{F}_L$ and $\mathcal{S}_L$ originally defined in \cite{Letourneau2023unconstrained} as
\begin{align}
\begin{split}
     \mathcal{F}_L :=
     \bigg\{   \bigg\{\mu_n:= & \sum_{l=1}^L \prod_{i=1}^D \mu_{i, n_i}^{(l)}  \bigg\}_{\lvert n \rvert \leq d}  \, : \, \{\mu_i^{(l)} \}_{i,l=1}^{D,L} \in \mathbb{R}^{(2d+1) \times D \times L},\, \\   
     & \, \mathcal{M}_d(\mu^{(l)}_i), \, \mathcal{M}_{d-1}(\mu^{(l)}_i ; (1-x_i^2)) \succeq 0 \; \forall \, i,l, \, \mu_{0} = 1 \bigg\},
\end{split}
\end{align}
and
\begin{align}
     \mathcal{S}_L &:= \left \{ \mu(\cdot) := \sum_{l=1}^L \prod_{i=1}^D  \mu_i^{(l)} (\cdot)  \, :  \,  \mu^{(l)}_i (\cdot) \in \mathbb{B}_1  \; \forall \; i,l, \, \mu\left( [-1,1]^D \right ) = 1\right \}
\end{align}
with sets of the form
\begin{equation}
    \tilde{\mathcal{F}}_L = \left\{ \bigg\{\mu_n:=  \sum_{l=1}^L \prod_{i=1}^D \mu_{i, n_i}^{(l)}  \bigg\}_{\lvert n \rvert \leq d}  \in \mathcal{F}_L \, : \, \gamma^{(j)} \cdot \phi(\mu^{(l)}) = 0 \; \forall \; j, l \right\}
\end{equation}
and
\begin{equation}
    \tilde{\mathcal{S}}_L = \left\{  \mu(\cdot) := \sum_{l=1}^L \prod_{i=1}^D  \mu_i^{(l)} (\cdot)  \in \mathcal{S}_L \, : \, \mathrm{supp}(\mu^{(l)}(\cdot) ) \subset \left \{ x \in [-1,1]^D : g^{(j)}(x) = 0 \right \} \; \forall \; j, l \right\},
\end{equation}
respectively, where $\tilde{\mathcal{F}}_L$ should be recognized as the feasible set of Problem \ref{eq:reformulationcontinuous_generalized} under the map $\phi(\cdot)$, while $
\tilde{\mathcal{S}}_L$ is a the set of normalized convex combinations of $L$ product measures supported over the algebraic set under consideration. Upon making such a substitution, the proof follows verbatim using Proposition \ref{prop:measure_extension_generalized} in lieu of Proposition 1 from \cite{Letourneau2023unconstrained} to establish a surjective relationship between $\tilde{\mathcal{S}}_L$ and $\tilde{\mathcal{F}}_L$.

\end{appendices}

\end{document}